\documentclass[a4paper, 12pt]{article}

\def\margin{2cm}
\usepackage[left=\margin,right=\margin,top=\margin,bottom=\margin]{geometry}

\usepackage{algorithmic, algorithm}

\usepackage{amsmath, amssymb, amsthm}
\usepackage{pstricks}
\usepackage{enumerate}

\usepackage{authblk}

\title{Rainbow Connection Number and Radius}
\author{Manu~Basavaraju}
\author{L.~Sunil~Chandran}
\author{Deepak~Rajendraprasad\thanks{Partially supported by Microsoft Research India - PhD Fellowship.}}
\author{Arunselvan~Ramaswamy}
\affil
{
	Department of Computer Science and Automation, \authorcr 
	Indian Institute of Science, \authorcr
	Bangalore -560012, India. \authorcr
	\{manu, sunil, deepakr, arunselvan\}@csa.iisc.ernet.in
}
\date{}

\theoremstyle{definition}
\newtheorem{definition}{Definition}

\theoremstyle{plain}
\newtheorem{theorem}{Theorem}
\newtheorem{lemma}[theorem]{Lemma}
\newtheorem{corollary}[theorem]{Corollary}

\theoremstyle{remark}

\newtheorem{example}[theorem]{Example}

\def\into{\rightarrow}
\def\is{\leftarrow}
\def\iso{\zeta}
\def\A{\mathcal{A}}
\def\B{\mathcal{B}}

\def\Q{\mathcal{Q}}
\def\-{\mbox{--}}
\newcommand{\bound}[2]{\sum_{i=1}^{#1}{\min\{2i+1, #2\}}}
\newcommand{\halff}[1]{\lfloor \frac{#1}{2} \rfloor} 
\newcommand{\halfc}[1]{\lceil \frac{#1}{2} \rceil} 

\begin{document}

\maketitle

\begin{abstract}
The {\em rainbow connection number}, $rc(G)$, of a connected graph $G$ is the minimum number of colours needed to colour its edges, so that every pair of its vertices is connected by at least one path in which no two edges are coloured the same. In this note we show that for every bridgeless graph $G$ with radius $r$, $rc(G) \leq r(r+2)$. We demonstrate that this bound is the best possible for $rc(G)$ as a function of $r$, not just for bridgeless graphs, but also for graphs of any stronger connectivity. It may be noted that for a general $1$-connected graph $G$, $rc(G)$ can be arbitrarily larger than its radius ($K_{1,n}$ for instance). We further show that for every bridgeless graph $G$ with radius $r$ and chordality (size of a largest induced cycle) $k$, $rc(G) \leq rk$. Hitherto, the only reported upper bound on the rainbow connection number of bridgeless graphs is $4n/5 - 1$, where $n$ is order of the graph \cite{caro2008rainbow}. 

It is known that computing $rc(G)$ is NP-Hard \cite{chakraborty2009hardness}. Here, we present a $(r+3)$-factor approximation algorithm which runs in $O(nm)$ time and a $(d+3)$-factor approximation algorithm which runs in $O(dm)$ time to rainbow colour any connected graph $G$ on $n$ vertices, with $m$ edges, diameter $d$ and radius $r$.
\end{abstract}


\section{Introduction}

An {\em edge colouring} of a graph is a function from its edge set to the set of natural numbers. A path in an edge coloured graph with no two edges sharing the same colour is called a {\em rainbow path}. An edge coloured graph is said to be {\em rainbow connected} if every pair of vertices is connected by at least one rainbow path. Such a colouring is called a {\em rainbow colouring} of the graph. The minimum number of colours required to rainbow colour a connected graph is called its {\em rainbow connection number}, denoted by $rc(G)$. For example, the rainbow connection number of a complete graph is $1$, that of a path is its length, and that of a tree is its number of edges. For a basic introduction to the topic, see Chapter $11$ in \cite{chartrand2008chromatic}.

The concept of rainbow colouring was introduced in \cite{chartrand2008rainbow}. It was shown in \cite{chakraborty2009hardness} that computing the rainbow connection number of a graph is NP-Hard. To rainbow colour a graph, it is enough to ensure that every edge of some spanning tree in the graph gets a distinct colour. Hence the order of the graph minus one is an upper bound for its rainbow connection number. Many authors view rainbow connectivity as one `quantifiable' way of strengthening the connectivity property of a graph \cite{caro2008rainbow,chakraborty2009hardness,krivelevich2010rainbow}. Hence tighter upper bounds on the rainbow connection number for graphs with higher connectivity have been a subject of investigation. The following are the results in this direction reported in literature: Let $G$ be a graph of order $n$. If $G$ is 2-edge-connected (bridgeless), then $rc(G) \leq 4n/5 -1$ and if $G$ is 2-vertex-connected, then $rc(G) \leq \min\{2n/3, n/2 + O(\sqrt{n})\}$ \cite{caro2008rainbow}. This was very recently improved in \cite{li2012rainbow}, where it was shown that if $G$ is $2$-vertex-connected, then $rc(G) \leq \lceil n/2 \rceil$, which is the best possible upper bound for the case. It also improved the previous best known upper bound of $3(n+1)/5$ for $3$-vertex connected graphs \cite{li2010rain3con}. It was shown in \cite{krivelevich2010rainbow} that $rc(G) \leq 20n/\delta$ where $\delta$ is the minimum degree of $G$. The result was improved in \cite{chandran2011raindom} where it was shown that $rc(G) \leq 3n/(\delta + 1) + 3$. Hence it follows that $rc(G) \leq 3n/(\lambda + 1) + 3$ if $G$ is $\lambda$-edge-connected and $rc(G) \leq 3n/(\kappa + 1) + 3$ if $G$ is $\kappa$-vertex-connected. It was shown in \cite{li2012rainbow} that the above bound in terms of edge connectivity is tight up to additive constants and that the bound in terms of vertex connectivity can be improved to $(2 + \epsilon)n/\kappa + 23/ \epsilon^2$, for any $\epsilon > 0$.

All the above upper bounds grow with $n$. The diameter of a graph, and hence its radius, are obvious lower bounds for its rainbow connection number. Hence it is interesting to see if there is an upper bound for rainbow connection number which is a function of radius or diameter alone. Such upper bounds were shown for some special graph classes in \cite{chandran2011raindom}. But, for a general graph, the rainbow connection number cannot be upper bounded by a function of $r$ alone. For instance, the star $K_{1,n}$ has radius $1$ but rainbow connection number $n$. In fact, it is easy to see that the number of bridges in a graph is also a lower bound on its rainbow connection number. Still, the question of whether such an upper bound exists for graphs with higher connectivity remains. Here we answer this question in the affirmative. In particular, we show that if $G$ is bridgeless, then $rc(G) \leq r(r+2)$ where $r$ is the radius of $G$ (Corollary \ref{cor:rainradius}). Moreover, we also demonstrate that the bound cannot be improved even if we assume stronger connectivity (Example \ref{ex:tight}). The technique presented in this paper of growing a connected multi-step dominating set was later extended in \cite{dong2011rainbridge} to show an upper bound for the rainbow connection number of a general connected graph in terms of its radius and  number of bridges.

Since the above bound is quadratic in $r$, we tried to see what additional restriction would give an upper bound which is linear in $r$. To this end, we show that if the size of isometric cycles or induced cycles in a graph is bounded independently of $r$, then the rainbow connection number is linear in $r$. In particular, we show that if $G$ is a bridgeless graph with radius $r$ and the size of a largest isometric cycle $\zeta$, then $rc(G) \leq r\zeta$ (Theorem \ref{thm:generalbound}). Since every isometric cycle is induced, it also follows that $rc(G) \leq rk$ where $k$ is the chordality (size of a largest induced cycle) of $G$ (Corollary \ref{cor:chordal}).

Since computing $rc(G)$ is NP-Hard \cite{chakraborty2009hardness}, it is natural to ask for approximation algorithms for rainbow colouring a graph. Our proof for the $r(r+2)$ bound is constructive and hence yields a $(r+2)$-factor approximation algorithm to rainbow colour any bridgeless graph $G$ of radius $r$. Note that $r$ is a lower bound on $rc(G)$ and hence the approximation factor. We show that this algorithm runs in $O(nm)$ time, where $n$ and $m$ are the number of vertices and edges of $G$ respectively. We also present an algorithm which has a smaller running time of $O(dm)$ but with a slightly poorer approximation ratio of $(d+2)$, where $d$ is the diameter of $G$. Both these algorithms are described in Section \ref{sec:bridgelessapprox}. Bridges in a connected graph can be found in $O(m)$ time \cite{tarjan1974note}. Contracting every bridge of a general connected graph gives a bridgeless graph and its rainbow colouring can be extended to the original graph by giving a new colour to every bridge. Using these ideas, we give a $(r+3)$-factor approximation algorithm which runs in $O(nm)$ time and a $(d+3)$-factor approximation algorithm which runs in $O(dm)$ time to rainbow colour any connected graph $G$ on $n$ vertices, with $m$ edges, diameter $d$ and radius $r$ (Section \ref{sec:generalapprox}).

\subsection{Preliminaries}
\label{sec:prelims}

All the graphs considered in this article are finite, simple and undirected. The {\em length} of a path $P$ is its number of edges and is denoted by $|P|$.  An edge in a connected graph is called a {\em bridge}, if its removal disconnects the graph. A connected graph with no bridges is called a {\em bridgeless} (or {\em $2$-edge-connected}) graph. If $S$ is a subset of vertices of a graph $G$, the subgraph of $G$ induced by the vertices in $S$ is denoted by $G[S]$. The graph obtained by contracting the set $S$ into a single vertex $v_S$ is denoted by $G/S$. The vertex set and edge set of $G$ are denoted by $V(G)$ and $E(G)$ respectively.

\begin{definition}
Let $G$ be a connected graph. The {\em distance} between two vertices $u$ and $v$ in $G$, denoted by $d_G(u,v)$ is the length of a shortest path between them in $G$. The {\em eccentricity} of a vertex $v$ is $ecc(v) := \max_{x \in V(G)}{d_G(v, x)}$. The {\em diameter} of $G$ is $diam(G) := \max_{x \in V(G)}{ecc(x)}$. The {\em radius} of $G$ is $rad(G) := \min_{x \in V(G)}{ecc(x)}$. The distance between a vertex $v$ and a set $S \subseteq V(G)$ is $d_G(v, S) := \min_{x \in S}{d_G(v,x)}$. The neighbourhood of $S$ is $N(S) := \{x \in V(G) | d_G(x, S) =1\}$.
\end{definition}

\begin{definition}
\label{defn:domination}
Given a graph $G$, a set $D \subseteq V(G)$ is called a {\em $k$-step dominating set} of $G$, if every vertex in $G$ is at a distance at most $k$ from $D$. Further if $G[D]$ is connected, then $D$ is called a {\em connected $k$-step dominating set} of $G$. 
\end{definition}

\begin{definition}
A subgraph $H$ of a graph $G$ is called {\em isometric} if the distance between any pair of vertices in $H$ is the same as their distance in $G$. The size of a largest isometric cycle in $G$ is denoted by $iso(G)$. 
\end{definition}

\begin{definition}
A graph is called {\em chordal} if it contains no induced cycles of length greater than $3$. The {\em chordality} of a graph $G$ is the length of a largest induced cycle in $G$.
\end{definition}

Note that every isometric cycle is induced and hence $iso(G)$ is at most the chordality of $G$. Also note that $3 \leq iso(G) \leq 2 \cdot diam(G) + 1$ for every bridgeless graph $G$.

\section{Upper Bounds for Bridgeless Graphs}

The most important idea in this note is captured in Lemma \ref{lem:domgrow} and all the upper bounds reported here will follow easily from it. The next important idea in this note, which is used in the construction of all the tight examples, is illustrated in Theorem \ref{thm:generalbound}. Before stating Lemma \ref{lem:domgrow}, we state and prove two small lemmas which are used in its proof. 

\begin{lemma}
\label{lem:isocycle}
For every edge $e$ in a graph $G$, any shortest cycle containing $e$ is isometric. 
\end{lemma}
\begin{proof}
Let $C$ be a shortest cycle containing $e$. For contradiction, assume that there exists at least one pair $(x, y) \in V(C) \times V(C)$ such that $d_G(x,y) < d_C(x,y)$. Choose $(x, y)$ to be one with minimum $d_G(x,y)$ among all such pairs. Let $P$ be a shortest $x\-y$ path in $G$. First we show that $P \cap C = \{x,y\}$. If $P \cap C$ contains some vertex $z \notin \{x,y\}$, then $d_G(x,z) + d_G(z,y) = d_G(x,y) < d_C(x,y) \leq d_C(x,z) + d_C(z,y)$. First equality follows since $P$ is a shortest $x\-y$ path, the strict inequality follows by assumption and the last is triangle inequality. Therefore, either $d_G(x,z) < d_C(x,z)$ or $d_G(y,z) < d_C(y,z)$. This contradicts the choice of $(x,y)$. Now it is easy to see that $P$ together with the segment of $C$ between $x$ and $y$ containing $e$ will form a cycle of length strictly smaller than $C$ and containing $e$. This contradicts the minimality of $C$. Hence $C$ is isometric.
\qed \end{proof}

\begin{definition}
Given a graph $G$ and a set $D \subset V(G)$, a {\em $D$-ear} is a path $P = (x_0, x_1, \ldots, x_p)$ in $G$ such that $P \cap D = \{x_0, x_p\}$. $P$ may be a closed path, in which case $x_0 = x_p$. Further, $P$ is called an {\em acceptable} $D$-ear if either $P$ is a shortest $D$-ear containing $(x_0, x_1)$ or $P$ is a shortest $D$-ear containing $(x_{p-1}, x_p)$. 
\end{definition}

\begin{lemma}
\label{lem:acceptable}
If $P$ is an acceptable $D$-ear in a graph $G$ for some $D \subset V(G)$, then $d_G(x,D) = d_P(x,D)$ for every $x \in P$. 
\end{lemma}
\begin{proof}
Without loss of generality, let $P = (x_0, x_1, \ldots, x_p)$ be a shortest $D$-ear containing $e = (x_0, x_1)$. Let $G' = G/D$ be the graph obtained by contracting $D$ into a single vertex $v_D$. It is easy to see that $P' = (v_D, x_1, x_2, \ldots, x_{p-1}, v_D)$ is a shortest cycle in $G'$ containing $e = (v_D, x_1)$. Hence by Lemma \ref{lem:isocycle}, $P'$ is isometric in $G'$. Now the result follows since $d_G(x,D) = d_{G'}(x,v_D)$ and $d_P(x,D) = d_{P'}(x,v_D)$. 
\qed \end{proof}

\begin{lemma}
\label{lem:domgrow}
If $G$ is a bridgeless graph, then for every connected $k$-step dominating set $D^k$ of $G$, $k \geq 1$, there exists a connected $(k-1)$-step dominating set $D^{k-1} \supset D^k$ such that 
$$rc(G[D^{k-1}]) \leq rc(G[D^k]) + \min\{2k + 1, \iso \},$$
where $\iso = iso(G)$.
\end{lemma}

\begin{proof}
Given $D^k$, we rainbow colour $G[D^k]$ with $rc(G[D^k])$ colours. Let $m = \min\{2k + 1, \iso \}$ and let $\A = \{a_1, a_2, \ldots \}$ and $\B = \{b_1, b_2, \ldots\}$ be two pools of colours, none of which are used to colour $G[D^k]$. A $D^k$-ear $P = (x_0, x_1, \ldots, x_p)$ will be called {\em evenly coloured} if its edges are coloured $a_1, a_2, \ldots, a_{\halfc{p}}, b_{\halff{p}}, \ldots, b_2, b_1$ in that order. We prove the lemma by constructing a sequence of sets $D^k = D_0 \subset D_1 \subset \cdots \subset D_t = D^{k-1}$ such that $D_{i+1} = D_i \cup P$, where $P$ is an acceptable $D^k$-ear and then colouring $G[D_{i+1}]$ in such a way that $P$ is evenly coloured using at most $m$ colours from $\A \cup \B$. In particular, this ensures that every $x \in D_{i} \backslash D^k$, $0 \leq i \leq t$, lies in an evenly coloured acceptable $D^k$-ear throughout the construction. 

If $N(D^k) \subset D_i$, then $D_i$ is a $(k-1)$-step dominating set and we stop the procedure by setting $t = i$. Otherwise pick any edge $e = (x_0, x_1) \in D^k \times (N(D^k) \backslash D_i)$ of $G$ and let $Q = (x_0, x_1, \ldots, x_q)$ be a shortest $D_k$-ear containing $e$. Such an ear always exists since $G$ is bridgeless. Let $x_l$ be the first vertex of $Q$ in $D_i$. If $x_l = x_q$, then evenly colour $Q$. Hence $P = Q$ is an evenly coloured acceptable $D^k$-ear. Otherwise $x_l$ is on some evenly coloured acceptable $D^k$-ear $P'$ added in an earlier iteration. By Lemma \ref{lem:acceptable}, $d_{P'}(x_l, D^k) = d_G(x_l, D^k)$. Hence the shorter segment $R$ of $P'$ (from $x_l$ to $D^k$) together with $L = (x_0, x_1, \ldots, x_l)$ is also an acceptable $D^k$-ear, $P = L \cup R$ containing $e$. Colour the edges of $L$ so that $P$ is evenly coloured. This is possible because (i) $R$ uses colours exclusively from one pool ($|R| \leq \lfloor |P'|/2 \rfloor$, since it is a shorter segment of $P'$) and (ii) $R$ forms a shorter segment of $P$ ($|L| \geq d_G(x_l, D^k) = |R|$, by Lemma \ref{lem:acceptable}). Hence the colouring of $R$ can be evenly extended to $L$. Set $D_{i+1} = D_i \cup P$. 

Firstly, we claim that at most $m$ new colours are used in the above procedure for constructing $D^{k-1}$ from $D^k$. Since $D^k$ is a $k$-step dominating set and since the $D^k$-ear $P = (x_0, x_1, \ldots, x_p)$ added in each iteration is acceptable, it follows that $|P| \leq 2k + 1$. Otherwise a middle vertex $x_{\halff{p}}$ of $P$ will be at a distance more than $k$ from $D^k$ (Lemma \ref{lem:acceptable}). Let $C$ be a shortest cycle containing $e = (x_0, x_1)$. $C$ exists since $G$ is bridgeless. By Lemma \ref{lem:isocycle}, $C$ is isometric and hence $|C| \leq \iso$. Further, $|P| \leq |C|$ since a sub-path of $C$ is a $D^k$-ear containing $e$. Thus $|P| \leq m = \min\{2k+1, \iso\}$ in every iteration. Hence all the new colours used in the procedure are from $\{a_1, \ldots, a_{\halfc{m}}\} \cup \{b_1, \ldots b_{\halff{m}}\}$, i.e., at most $m$ new colours are used. 

Next, we claim that the $G[D^{k-1}]$ constructed this way is rainbow connected. Any pair $(x,y) \in D^k \times D^k$, is rainbow connected in $G[D^k]$. For any pair $(x, y) \in (D^{k-1} \backslash D^k) \times D^k$, let $P = (x_0, x_1, \ldots, x_i = x, \ldots, x_p)$ be the evenly coloured (acceptable) $D^k$-ear containing $x$. Joining $(x = x_i, x_{i+1}, \ldots, x_p)$ with a $x_p \- y$ rainbow path in $G[D^k]$ gives a $x \- y$ rainbow path. For any pair $(x, y) \in (D^{k-1} \backslash D^k) \times (D^{k-1} \backslash D^k)$, let $P = (x_0, x_1, \ldots, x_i = x, \ldots, x_p)$ and $Q = (y_0, y_1, \ldots, y_j = y, \ldots, y_q)$ be evenly coloured (acceptable) $D^k$-ears containing $x$ and $y$ respectively. Recall that the vertices of $P$ and $Q$ are ordered in such a way that their first halves get colours from Pool $\A$. We consider the following $4$ cases. If $i \leq \halff{p}$ and $j > \halff{q}$, then joining $(y = y_{j}, y_{j+1} \ldots, y_q)$ (which is $\B$-coloured) to the $y_q \- x_0$ rainbow path in $G[D^k]$ followed by $(x_0, x_1, \ldots, x_i = x)$ (which is $\A$-coloured) gives a $x \- y$ rainbow path. Case when $i > \halff{p}$ and $j \leq \halff{q}$ is similar. When $i \leq \halff{p}$ and $j \leq \halff{q}$ check if $i \leq j$. If yes, join $(y = y_j, y_{j+1}, \ldots, y_q)$ (which uses colours from $\{a_l \in \A : l \geq j+1\} \cup \B$) to the $y_q \- x_0$ rainbow path in $G[D^k]$ followed by $(x_0, x_1, \ldots, x_i = x)$ (which uses colours from $\{a_l \in \A : l \leq i\}$) to get an $x\-y$ rainbow path. If $i > j$, then do the reverse. In the final case, when $i > \halff{p}$ and $j > \halff{q}$ check if $q-j \leq p-i$. If yes, join $(y = y_j, y_{j+1}, \ldots, y_q)$ (which uses colours from $\{b_l \in \B : l \leq q-j \}$ to the $y_q \- x_0$ rainbow path in $G[D^k]$ followed by $(x_0, x_1, \ldots, x_i = x)$ (which uses colours from $\A \cup \{b_l \in \B : l \geq p-i+1\}$) to get an $x\-y$ rainbow path. If $q-j > p-i$, then do the reverse. Any edge in $G[D^{k-1}]$ left uncoloured by the procedure can be assigned any used colour to complete the rainbow colouring.
\qed \end{proof}

\begin{theorem}
\label{thm:generalbound}
For every bridgeless graph $G$, 
$$ rc(G) \leq \bound{r}{\iso} \leq r\iso,$$
where $r$ is the radius of $G$ and $\iso = iso(G)$.
 
Moreover, for every two integers $r \geq 1$, and $3 \leq \iso \leq 2r + 1$, there exists a bridgeless graph $G$ with radius $r$ and $iso(G) = \iso$ such that $rc(G) = \bound{r}{\iso}$. 
\end{theorem}
\begin{proof}
If $u$ is a central vertex of $G$, i.e., $ecc(u) = r$, then $D^r = \{u\}$ is an $r$-step dominating set in $G$ and $rc(G[D^r]) = 0$. The only $0$-step dominating set in $G$ is $V(G)$.  Hence, repeated application of Lemma \ref{lem:domgrow} gives the upper bound

To construct a tight example for a given $r \geq 1$ and $3 \leq \iso \leq 2r+1$, consider the graph $H_{r,\iso}$ in Figure \ref{fig:cyclechain}. Note that (i) $H_{r,\iso}$ is bridgeless, (ii) the size of largest isometric cycle in $H_{r,\iso}$ is $\iso$, and (iii) $ecc(u) = r$ for any $\iso \leq 2r+1$.

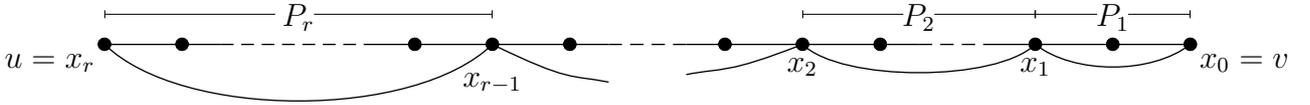
\begin{figure}[ht]
\begin{center}
\psset{xunit=0.06\textwidth}
\begin{pspicture}(-1,-0.5)(15,0.5)
	\psset{labelsep=5pt,linewidth=0.5pt}
	\psdots[dotstyle=o, dotsize=5pt,fillstyle=solid, fillcolor=black]
		(0,0)(1,0)(4,0)(5,0)(6,0)(8,0)(9,0)(10,0)(12,0)(13,0)(14,0)
	
	\psline{-}(0,0)(1.5,0)
	\psline[linestyle=dashed]{-}(1.5,0)(3.5,0)
	\psline{-}(3.5,0)(6.5,0)
	\psline[linestyle=dashed]{-}(6.5,0)(7.5,0)
	\psline{-}(7.5,0)(10.5,0)
	\psline[linestyle=dashed]{-}(10.5,0)(11.5,0)
	\psline{-}(11.5,0)(14,0)
	
	\psbezier(0,0)(1,-1)(4,-1)(5,0)(6,-0.5)(6,-0.4)(6.5,-0.5)
	\psbezier(7.5,-0.4)(8,-0.3)(8,-0.4)(9,0)(9.5,-0.5)(11.5,-0.5)(12,0)(12.5,-0.4)(13.5,-0.4)(14,0)
	
	\uput[dl](0,0){$u = x_r$}
	
	\psline[linewidth=0.2pt]{|-}(0,0.4)(2.2,0.4)
	\uput[u](2.5,0){$P_{r}$}
	\psline[linewidth=0.2pt]{-|}(2.7,0.4)(5,0.4)

	\uput[d](5,-0.2){$x_{r-1}$}
	\uput[d](9,0){$x_2$}

	\psline[linewidth=0.2pt]{|-}(9,0.4)(10.3,0.4)
	\uput[u](10.5,0){$P_2$}
	\psline[linewidth=0.2pt]{-|}(10.7,0.4)(12,0.4)
	
	\uput[d](12,0){$x_1$}
	
	\psline[linewidth=0.2pt]{|-}(12,0.4)(12.8,0.4)
	\uput[u](13,0){$P_1$}
	\psline[linewidth=0.2pt]{-|}(13.2,0.4)(14,0.4)

	\uput[dr](14,0){$x_0 = v$}

\end{pspicture}
\end{center}
\caption{Graph $H_{r, \iso}$. Every $P_i$ is a $x_{i-1}$--$x_i$ path of length $|P_i| = \min\{2i, \iso-1\}$.}
\label{fig:cyclechain}
\end{figure}

Let $m := \bound{r}{\iso}$. Construct a graph $G$ by taking $m^r + 1$ graphs $\{H^j\}_{j=0}^{m^r}$ where $V(H^j) = \{x^j : x \in V(H_{r,\iso}) \}$ and $E(H^j) = \{\{x^j, y^j\} : \{x,y\} \in E(H_{r,\iso}) \}$. Identify the vertex $u^j$ as common in every copy ($u = u^j, 0 \leq j \leq m^r$). It can be easily verified that (i) $G$ is bridgeless (ii) $rad(G) = r$ and (ii) size of the largest isometric cycle in $G$ is $\iso$. Hence, by first part of this theorem, $k := rc(G) \leq m$. In any edge colouring $c : E(G) \into \{1, 2, \ldots, k\}$ of $G$, each $r$-length $u\-v^j$ path can be coloured in at most $k^r$ different ways. By pigeonhole principle, there exist $p \neq q$, $0 \leq p, q \leq m^r$ such that $c(e_i^p) = c(e_i^q), 1 \leq i \leq r$ where $e_i^j = (x_{i-1}^j, x_i^j)$. Consider any rainbow path $R$ between $v^p$ and $v^q$. For every $1 \leq i \leq r$, $|R \cap \{e_i^p, e_i^q\}| \leq 1$ (since $c(e_i^p) = c(e_i^q)$) and hence $P_i^j \subset R$ for some $j \in \{p,q\}$. Thus $|R| \geq \sum_{i=1}^r{(1+|P_i|)} = m$. Hence $k \geq m$ and $G$ gives the required tight example. 
\qed \end{proof}

\begin{corollary}
\label{cor:rainradius}
For every bridgeless graph $G$ with radius $r$, 
$$ rc(G) \leq r(r+2).$$
Moreover, for every integer $r \geq 1$, there exists a bridgeless graph with radius $r$ and $rc(G) = r(r+2)$. 
\end{corollary}
\begin{proof}
Noting that $\min\{2i+1, \iso\} \leq 2i+1$, the upper bound follows from Theorem \ref{thm:generalbound}. The tight examples are obtained by setting $\iso = 2r+1$ in the tight examples for Theorem \ref{thm:generalbound}
\qed \end{proof}

A natural question at this stage is whether the upper bound of $r(r+2)$ can be improved if we assume a stronger connectivity for $G$. But the following example shows that it is not the case.

\begin{example}
[{\em Construction of a $\kappa$-connected graph of radius $r$ whose rainbow connection number is $r(r+2)$ for any two given integers $\kappa, r \geq 1$}]
\label{ex:tight}

Let $s(0) := 0$, $s(i) := 2\sum_{j=r}^{r-i+1}j$ for $1 \leq i \leq r$ and $t := s(r) = r(r+1)$. Let $V = V_0 \uplus V_1 \uplus \cdots \uplus V_t$ where $V_i = \{x_{i,0}, x_{i,1}, \ldots, x_{i,\kappa-1}\}$ for $0 \leq i \leq t-1$ and $V_t = \{x_{t,0}\}$. Construct a graph $X_{r,\kappa}$ on $V$ by adding the following edges. 
$E(X) = \{ \{x_{i,j}, x_{i',j'}\} : |i-i'| \leq 1\} \cup \{ \{x_{s(i),0}, x_{s(i+1),0} \} : 0 \leq i \leq r-1\}.$
Figure \ref{fig:cyclechain2} depicts $X_{3,2}$.

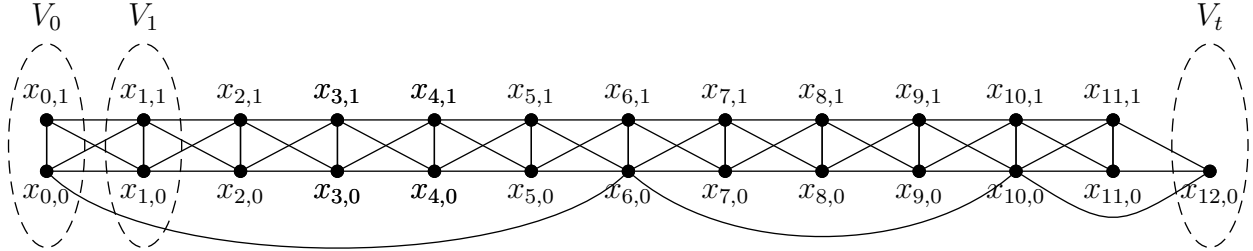
\begin{figure}[ht]
\begin{center}
\psset{xunit=0.075\textwidth, yunit=0.04\textwidth}
\begin{pspicture}(0,-1)(12,3)
	\psset{labelsep=5pt,linewidth=0.5pt}
	\psdots[dotstyle=o, dotsize=5pt,fillstyle=solid, fillcolor=black]
		(0,0)(1,0)(2,0)(3,0)(4,0)(5,0)(6,0)(7,0)(8,0)(9,0)(10,0)(11,0)(12,0)
		(0,1)(1,1)(2,1)(3,1)(4,1)(5,1)(6,1)(7,1)(8,1)(9,1)(10,1)(11,1)

	\psline{-}(0,0)(0,1)(1,0)(1,1)(2,0)(2,1)(3,0)(3,1)(4,0)(4,1)(5,0)(5,1)(6,0)(6,1)(7,0)(7,1)(8,0)(8,1)(9,0)(9,1)(10,0)(10,1)(11,0)(11,1)
	\psline{-}(0,1)(0,0)(1,1)(1,0)(2,1)(2,0)(3,1)(3,0)(4,1)(4,0)(5,1)(5,0)(6,1)(6,0)(7,1)(7,0)(8,1)(8,0)(9,1)(9,0)(10,1)(10,0)(11,1)(11,0)
	\psline{-}(0,0)(12,0)(11,1)(0,1)
	\psbezier(0,0)(1,-2)(5,-2)(6,0)(7,-1.7)(9,-1.7)(10,0)(11,-1.2)(11,-1.2)(12,0)
	
	\psellipse[linestyle=dashed](0,0.5)(0.4,2)
	\uput[u](0,2.5){$V_0$}
	\psellipse[linestyle=dashed](1,0.5)(0.4,2)
	\uput[u](1,2.5){$V_1$}
	\psellipse[linestyle=dashed](12,0.5)(0.4,2)
	\uput[u](12,2.5){$V_t$}

	\uput[d](0,0){$x_{0,0}$}
	\uput[d](1,0){$x_{1,0}$}
	\uput[d](2,0){$x_{2,0}$}
	\uput[d](3,0){$x_{3,0}$}
	\uput[d](3,0){$x_{3,0}$}
	\uput[d](4,0){$x_{4,0}$}
	\uput[d](4,0){$x_{4,0}$}
	\uput[d](5,0){$x_{5,0}$}
	\uput[d](6,0){$x_{6,0}$}
	\uput[d](7,0){$x_{7,0}$}
	\uput[d](8,0){$x_{8,0}$}
	\uput[d](9,0){$x_{9,0}$}
	\uput[d](10,0){$x_{10,0}$}
	\uput[d](11,0){$x_{11,0}$}
	\uput[d](12,0){$x_{12,0}$}

	\uput[u](0,1){$x_{0,1}$}
	\uput[u](1,1){$x_{1,1}$}
	\uput[u](2,1){$x_{2,1}$}
	\uput[u](3,1){$x_{3,1}$}
	\uput[u](3,1){$x_{3,1}$}
	\uput[u](4,1){$x_{4,1}$}
	\uput[u](4,1){$x_{4,1}$}
	\uput[u](5,1){$x_{5,1}$}
	\uput[u](6,1){$x_{6,1}$}
	\uput[u](7,1){$x_{7,1}$}
	\uput[u](8,1){$x_{8,1}$}
	\uput[u](9,1){$x_{9,1}$}
	\uput[u](10,1){$x_{10,1}$}
	\uput[u](11,1){$x_{11,1}$}

\end{pspicture}
\end{center}
\caption{Graph $X_{3,2}$. Note: (i) $X_{3,2}$ is $2$-connected and (ii) $ecc(x_{0,0}) = 3$.} 
\label{fig:cyclechain2}
\end{figure}

Let $m = r(r+2)$. Construct a new graph $G$ by taking $m^r + 1$ copies of $X_{r,k}$ and identifying the vertices in $V_0$ as common in every copy. It is easily seen that $G$ is $\kappa$-connected and has a radius $r$ with $x_{0,0}$ as the central vertex. By arguments similar to those in the tight examples for Theorem \ref{thm:generalbound}, we can see that $rc(G) = m$.
\end{example}

\begin{corollary}
\label{cor:chordal}
For every bridgeless graph $G$ with radius $r$ and chordality $k$, 
$$ rc(G) \leq \bound{r}{k} \leq rk.$$
Moreover, for every two integers $r \geq 1$ and $3 \leq k \leq 2r+1$, there exists a bridgeless graph $G$ with radius $r$ and chordality $k$ such that $rc(G) = \bound{r}{k}$. 
\end{corollary}
\begin{proof}
Since every isometric cycle is an induced cycle, the chordality of a graph is at least the size of its largest isometric cycle. i.e, $k \geq \iso$.  Hence the upper bound follows from that in Theorem \ref{thm:generalbound}. The tight example demonstrated in Theorem \ref{thm:generalbound} suffices here too.
\qed \end{proof}

This generalises a result from \cite{chandran2011raindom} that the rainbow connection number of any bridgeless chordal graph is at most three times its radius.

\section{Approximation Algorithms}

\subsection{Bridgeless Graphs}
\label{sec:bridgelessapprox}

Throughout this section, $G$ will be a bridgeless graph with $n$ vertices, $m$ edges, diameter $d$ and radius $r$. A set $S \subset V(G)$ will be called {\em rainbow coloured} under a partial edge colouring of $G$ if every pair of vertices in $S$ is connected by a rainbow path in $G[S]$.  

\subsubsection{$O(nm)$ time $(r+2)$-factor Approximation Algorithm} 
\label{sec:radiusapprox}

Corollary \ref{cor:rainradius} was proved by demonstrating a colouring procedure which assigns a rainbow colouring to any bridgeless graph of radius $r$ using at most $r(r+2)$ colours. Since the proof is constructive, it automatically gives us an algorithm for rainbow colouring $G$. Since $r$ is a lower bound on rainbow connection number, this is a $(r+2)$-factor approximation algorithm. The procedure starts by identifying a central vertex in the graph. This can be done by computing the eccentricity of every vertex using a Breadth First Search (BFS) rooted at it. Thus the time complexity for finding the central vertex in any connected graph is $O(nm)$. The acceptable ears to be coloured in each step can be found using a BFS rooted at the selected vertex in $N(D^k)$ on a subgraph of $G$ and hence takes $O(m)$ running time on any connected graph. Since we do not start the BFS more than once from any vertex, the total running time for finding all the acceptable ears that gets coloured is $O(nm)$. The colouring of a selected acceptable ear takes a time proportional to the number of uncoloured edges in that ear. Moreover, each edge is coloured only once by the algorithm. Hence the total effect of colour assignments on the algorithm's running time is $O(m)$.  Thus the total running time for the algorithm is $O(nm)$.

Next we present an algorithm which has a smaller running time of $O(dm)$ but a slightly poorer approximation ratio of $(d+2)$.

\subsubsection{$O(dm)$ time $(d+2)$-factor Approximation Algorithm} 
\label{sec:diameterapprox}

To the best of our knowledge, there is no known algorithm to find a central vertex of a bridgeless graph in a time significantly smaller than $\Theta(nm)$. Hence we start the procedure by picking any arbitrary vertex $v$ of $G$ ($O(1)$ time). Since $ecc(v) \leq d$, this is connected $d$-step dominating set of $G$. Hence, by repeated application of Lemma \ref{lem:domgrow}, we can grow the trivially rainbow coloured connected $d$-step dominating set $D^d = \{v\}$ to a rainbow coloured connected $0$-step dominating set $D^0 = V(G)$ using  at most $d(d+2)$ colours. So if we can grow a rainbow coloured connected $k$-step dominating set $D^k$ to a rainbow coloured connected $(k-1)$-step dominating set $D^{k-1}$ in $O(m)$ time, then we can complete the rainbow colouring of $G$ using $d(d+2)$ colours in $O(dm)$ time. Since $d$ is a lower bound on rainbow connection number this gives a $(d+2)$-factor approximation algorithm.

In the proof of Lemma \ref{lem:domgrow}, given a rainbow coloured connected $k$-step dominating set $D^k$, we pick any edge $e = (x_0,x_1)$ with $x_0 \in D^k$ and $x_1$ being an uncaptured vertex in $N(D^k)$. Next, we find an acceptable ear containing $e$ and evenly colour that ear. When every vertex in $N(D^k)$ is captured this way, we have a rainbow coloured connected $(k-1)$-step dominating set $D^{k-1}$ in hand. It is easy to see that, once an acceptable ear is found and the colours (if any) of its end edges are known, it can be evenly coloured in a time proportional to number of uncoloured edges in that ear. Since no edge is coloured more than once by the algorithm, the total running time for the colouring subroutine (once the acceptable ears are found) is only $O(m)$. Hence if we can capture every vertex in $N(D^k)$ using acceptable ears in $O(m)$ time, we can construct the required $D^{k-1}$ from the given $D^k$ in $O(m)$ time. This is precisely what Algorithm \ref{alg:domgrow} achieves. 

Algorithm \ref{alg:domgrow} accepts a partially edge coloured bridgeless graph $G$, a rainbow coloured connected $k$-step dominating set $D^k$ in $G$ and two pools of colours $\A = \{a_1, a_2, \ldots, a_{k+1}\}$ and $\B = \{b_1, b_2, \ldots, b_k \}$ not used in colouring $G[D^k]$. It returns a $(k-1)$-step dominating set $D^{k-1}$ of vertices and colours a subset of $E(G[D^{k-1}]) \setminus E(G[D^k])$ using colours from $\A \cup \B$ such that $G[D^{k-1}]$ is rainbow coloured. It achieves the same by running a single BFS on $G \setminus E(G[D^k])$ with the BFS queue initialised with $D^k$ and maintaining enough side information to detect meetings which result in acceptable ears. Once an acceptable ear is found, that ear is evenly coloured using colours from pools $\A$ and $\B$. The procedure terminates once every edge is examined and hence runs in $O(m)$ time.

\subsubsection*{Side information associated with each vertex $v$ in Algorithm \ref{alg:domgrow}}

\begin{description}

\item[$Parent$:] For each vertex $v$ visited by the BFS, $Parent(v)$ points to parent vertex of $v$ in the BFS forest. It is initialised to $\emptyset$ for all vertices.

\item[$ParentEdgeColour$:] For each new vertex $v$ captured by the algorithm ($v \in D^{k-1} \setminus D^k$), $ParentEdgeColour(v)$ holds the colour assigned to the edge $(v, Parent(v))$ by the algorithm. It is also initialised to $\emptyset$ for all vertices. This information is updated for the vertices of an acceptable ear when it gets evenly coloured during the algorithm. Note that it is only a temporary and partial information of the colourings effected in one run of the algorithm which is used to make an instant check of whether a vertex has been already captured by an evenly coloured acceptable ear and to detect the colour pool used. The colouring subroutine also encodes every colour assignment into the adjacency list of $G$ and that is what is finally returned. 

\item[$Foot$:] For each vertex $v$ visited by the BFS, $Foot(v)$ is the ordered pair of last two vertices in the BFS path from $v$ to $D^k$. It is set to $\emptyset$ for all vertices in the initial queue $D^k$.

\end{description}

\begin{algorithm}
\caption{Construct and rainbow colour $D^{k-1}$ given rainbow coloured $D^{k}$}
\label{alg:domgrow}
\begin{algorithmic}
\REQUIRE 
$G$ is a partially edge coloured bridgeless graph. $D^k$ is a rainbow coloured connected $k$-step dominating set in $G$. $\A = \{a_1, a_2, \ldots, a_{k+1}\}$ and $\B = \{b_1, b_2, \ldots, b_k \}$ are two pools of colours not used to colour $G[D^k]$. 
\ENSURE 
$D^{k-1} \supset D^k$ is a rainbow coloured connected $(k-1)$-step dominating set in $G$ such that new colours used are from $\A \cup \B$.
\FOR{each $u \in G$}
	\STATE $Parent(u) \is \emptyset$, $ParentEdgeColour(u) \is \emptyset$, $Foot(u) \is \emptyset$
\ENDFOR
\STATE Flush($\Q$), Enqueue($\Q$, $D^k$), $D^{k-1} \is D^k$
\REPEAT
	\STATE $u \is Dequeue(\Q)$
	\FOR{each vertex $v \in N(u) \cap V(G \setminus D^{k})$}
		\IF[$v$ is an unvisited vertex]{$Foot(v) = \emptyset$}
			\IF[$u \in D^k$]{$Foot(u) = \emptyset$}
				\STATE $Foot(v) \is (v,u)$
			\ELSE
				\STATE $Foot(v) \is Foot(u)$
			\ENDIF
			\STATE $Parent(v) \is u$, $Enqueue(\Q, v)$
		\ELSIF[we have found an acceptable $D^k$-ear]{$Foot(v) \neq Foot(u)$}
			\IF[$u \in D^k$]{$Foot(u) = \emptyset$}
				\STATE $u_0 = u$, $c_u = \emptyset$ 
			\ELSE[$u_0$ will hold the vertex of $Foot(u)$ in $D^k$ and $c_u$ will hold the colour of $Foot(u)$]
				\STATE $(u_1, u_0) \is Foot(u)$, $c_u \is ParentEdgeColour(u_1)$
			\ENDIF
			\STATE $(v_1, v_0) \is Foot(v)$, $c_v \is ParentEdgeColour(v_1)$
			\IF[$u_1$ or $v_1$ is an uncaptured vertex in $N(D^k)$]{$c_u = \emptyset$ \OR $c_v = \emptyset$}
				\STATE $P \is u_0 T u\-v Tv_0$ where $x T y$ is the unique path from $x$ to $y$ in the BFS forest under construction. \COMMENT{Path $P$ is an acceptable $D^k$ ear some of whose edges are still uncoloured}
				\STATE $p \is |P|$ \COMMENT{length of $P$} 
				\IF{$c_u = a_1$ \OR $c_v = b_1$ \OR $c_u = c_v = \emptyset$}
					\STATE The uncoloured edges of $P$ are coloured so that the edges of $P$ get the colours $a_1, a_2, \ldots, a_{\halfc{p}}, b_{\halff{p}}, \ldots, b_2, b_1$ in that order. 
				\ELSE
					\STATE The uncoloured edges of $P$ are coloured so that the edges of $P$ get the colours $b_1, b_2, \ldots, b_{\halff{p}}, a_{\halfc{p}}, \ldots, a_2, a_1$ in that order. 
				\ENDIF 
				\STATE $D^{k-1} \is D^{k-1} \cup P$
			\ENDIF 	
		\ENDIF
\ENDFOR
\UNTIL{$\Q$ is empty}

\end{algorithmic}
\end{algorithm}

\subsection{General Connected Graphs}
\label{sec:generalapprox}

In this section, $G$ will be a connected graph with $n$ vertices, $m$ edges, diameter $d$, radius $r$ and $b$ bridges. Let $G'$ be the graph obtained by contracting every bridge of $G$. The diameter (radius) of $G'$ is at most $d$ ($r$). We can extend a rainbow colouring of $G'$ to $G$ by giving a new colour to every bridge of $G$. Hence $rc(G) \leq rc(G') + b$. We can find all the bridges in a connected graph in $O(m)$ time \cite{tarjan1974note}. Now, using the algorithm in Section \ref{sec:radiusapprox} to colour $G'$, we can colour $G$ using at most $r(r+2) + b$ colours in $O(nm)$ time. Since $r(r+2) + b \leq \max\{r, b\}(r+3)$ and since $\max\{r,b\}$ is a lower bound on $rc(G)$, we immediately have a $(r+3)$-factor $O(nm)$ approximation algorithm to rainbow colour any connected graph.

Similarly by combining an $O(m)$ algorithm to find every bridge of $G$ with the algorithm in Section \ref{sec:diameterapprox} gives an $O(dm)$ algorithm to rainbow colour $G$ using $d(d+2) + b$ colours. Since  $d(d+2) + b \leq \max\{d, b\}(d+3)$ and since $\max\{d,b\}$ is a lower bound on $rc(G)$, we immediately have a $(d+3)$-factor $O(dm)$ approximation algorithm to rainbow colour any connected graph.

\bibliographystyle{plain}


\end{document}